\documentclass[1p]{elsarticle}
\usepackage{amssymb}
\usepackage{amsfonts}

\newtheorem{theorem}{Theorem}[section]

\newtheorem{claim}[theorem]{Claim}

\newtheorem{conjecture}[theorem]{Conjecture}
\newtheorem{corollary}[theorem]{Corollary}

\newtheorem{observation}[theorem]{Observation}
\newtheorem{problem}[theorem]{Problem}

\newproof{proof}{Proof}

\begin{document}
\title
{On decomposing graphs of large minimum degree into locally irregular subgraphs}

\author{Jakub Przyby{\l}o\fnref{grantJP,fn1}}
\ead{jakubprz@agh.edu.pl, phone: 048-12-617-46-38,  fax: 048-12-617-31-65}
\fntext[grantJP]{Supported by the National Science Centre, Poland, grant no. 2014/13/B/ST1/01855.}
\fntext[fn1]{Partly supported by the Polish Ministry of Science and Higher Education.}
\address{AGH University of Science and Technology, al. A. Mickiewicza 30, 30-059 Krakow, Poland}

\begin{abstract}
A \emph{locally irregular graph} is a graph whose adjacent vertices have distinct degrees.
We say that a graph $G$ \emph{can be decomposed into} $k$ \emph{locally irregular subgraphs} if
its edge set may be partitioned into $k$ subsets each of which induces a locally irregular subgraph in $G$.
It has been conjectured that apart from the family of exceptions which admit no such decompositions,
i.e., odd paths, odd cycles and a special class of graphs of maximum degree $3$,
every connected graph can be decomposed into $3$ locally irregular subgraphs.
Using a combination of a probabilistic approach and some known theorems on degree constrained subgraphs of a given graph,
we prove this to hold for graphs of sufficiently large minimum degree, $\delta(G)\geq 10^{10}$.
This problem is strongly related to edge colourings distinguishing neighbours
by the pallets of their incident colours and to 1-2-3 Conjecture.
In particular, the contribution of this paper constitutes a strengthening of
a result of Addario-Berry, Aldred,
Dalal and Reed [J. Combin. Theory Ser. B 94 (2005) 237-244].
\end{abstract}


\maketitle

\section{Introduction}
All graphs considered are simple and finite.
We follow~\cite{Diestel} for the notations and terminology not defined here.
A \emph{locally irregular} graph is a graph
whose every vertex
has the degree distinct from the degrees of all of its neighbours.
In other words, it is a graph in which the adjacent vertices have distinct degrees.
Motivated by a few well known problems in
edge colourings and labelings,
we investigate a (non-evidently)  related concept of decompositions of graphs into
locally irregular subgraphs.
More precisely, we say that a graph $G=(V,E)$ can be \emph{decomposed into} $k$ \emph{locally irregular subgraphs} if
its edge set may be partitioned into $k$ subsets each of which induces a locally irregular subgraph, i.e.,
$E=E_1\cup E_2\cup\ldots\cup E_k$ with $E_i\cap E_j=\emptyset$ for $i\neq j$ and $H_i:=(V,E_i)$ is locally irregular for $i=1,2,\ldots,k$.
Naturally, instead of decomposing the graph $G$, we may alterably
paint its edges with $k$ colours, say $1,2,\ldots,k$,
so that every colour class induces a locally irregular subgraph in $G$.
Such colouring is called a \emph{locally irregular} $k$-\emph{edge colouring}
of $G$.
Equivalently it is just an edge colouring such that if
an edge $uv\in E$ has colour $i\in \{1,2,\ldots,k\}$ assigned under it,
then the numbers of edges coloured with $i$ incident with $u$ and $v$ must be distinct.
As a weakening of such property, we may require $u$ and $v$ to
to differ in frequencies of any of the colours from $\{1,2,\ldots,k\}$,
not specifically the colour $i$ (which is assigned to $uv$).
In other words, we may wish the adjacent vertices to have distinct multisets of their incident colours under a colouring $c:E\to \{1,2,\ldots,k\}$.
We call such $c$ a \emph{neighbour multiset distinguishing} $k$-\emph{edge colouring} then.
Such variant of edge colourings has in fact already been investigated by
Addario-Berry et al. in~\cite{Louigi2}.
They proved that $4$ colours are always sufficient to construct such a colouring for every graph containing no isolated edges,
and provided the following improvement.
\begin{theorem}\label{multisets_th_2}
There exists a neighbour multiset distinguishing $3$-edge colouring of every graph $G$
of minimum degree $\delta\geq 10^3$.
\end{theorem}

Their research was motivated by the so called 1-2-3 Conjecture due to Karo\'nski, {\L}uczak and Thomason~\cite{123KLT},
yet another concept introducing `local irregularity' into a graph.
Let $c:E\to\{1,2,\ldots,k\}$ be an edge colouring of $G$ with positive integers.
For every vertex $v$ we
denote by $s_c(v):=\sum_{u\in N(v)}c(uv)$
the sum of its incident colours and call it the \emph{weighted degree} of $v$.
We say that $c$ is a \emph{neighbour sum distinguishing} $k$-\emph{edge colouring} of $G$
if $s_c(u)\neq s_c(v)$
for all adjacent vertices $u,v$ in $G$.
Equivalently, instead of assigning integers from $\{1,2,\ldots,k\}$ to the edges, one might strive to
multiply them the corresponding numbers of times
in order to create a \emph{locally irregular multigraph} of $G$, i.e., a multigraph
in which the adjacent vertices have distinct degrees.
In~\cite{123KLT} Karo\'nski et al.
posed the following elegant problem.
\begin{conjecture}[\textbf{1-2-3 Conjecture}]\label{123Conjecture}
There exists a neighbour sum distinguishing $3$-edge colouring of every graph $G$ containing no isolated edges.
\end{conjecture}
Thus far it is known that a neighbour sum distinguishing $5$-edge colouring exists
for every graph without isolated edges, see~\cite{KalKarPf_123}.
In fact our interest in locally irregular graphs originated from 1-2-3 Conjecture via the following easy observation from~\cite{LocalIrreg_1}.
\begin{observation}\label{equivEdge_2}
If $G$ is a regular graph, then there exists a neighbour sum distinguishing $2$-edge colouring of $G$
if and only if it can be decomposed into $2$ locally irregular subgraphs.
\end{observation}
It is
worth noting in this context that asymptotically almost surely a random $d$-regular graph
can be decomposed into $2$ locally irregular subgraphs for every constant $d>d_0$, where $d_0$ is some large number, see~\cite{LocalIrreg_1}.
In the same paper the authors investigate a special family $\mathfrak{T}$ of graphs of maximum degree (at most) $3$
whose every member might be constructed from a triangle by repeatedly performed the following operation:
choose a triangle with a vertex
of degree $2$ in our constructed graph and append to this vertex either
a hanging path of even length or a hanging path of odd length with a triangle glued to its other end.
They posed the following conjecture.
\begin{conjecture}\label{BBPWConjecture1}
Every connected graph $G$ which
does not belong to $\mathfrak{T}$ and is not an odd length path nor an odd length cycle
can be decomposed
into $3$ locally irregular subgraphs.
\end{conjecture}
The graphs excluded in the conjecture above were also proven to be the only connected
graphs which do not admit decompositions into any number of locally irregular subgraphs.
This conjecture was verified in~\cite{LocalIrreg_1} for some classes of graph, e.g.,
complete graph, complete bipartite graphs, trees, cartesian products of graphs with the desired property
(hypercubes for instance), and for regular graphs with large degrees.

The main result of this paper is the following strengthening of Theorem~\ref{multisets_th_2},
which confirms Conjecture~\ref{BBPWConjecture1} for graphs of sufficiently large minimum degree.
\begin{theorem}\label{main_Th_general_graphs}
Every graph $G$ with minimum degree $\delta\geq 10^{10}$
can be decomposed into three locally irregular subgraphs.
\end{theorem}
Its proof combines a probabilistic approach with some known theorems on degree constrained subgraphs.

To exemplify the fact that the two graph invariants representing the minimum
numbers of colors necessary to create a neighbour multiset distinguishing edge colouring and
a locally irregular edge colouring, resp., are indeed distinct,
let us consider a graph constructed as follows.
Take a single edge, say $uv$, and append two hanging paths of length $2$ to the vertex $u$ and
another two hanging paths of length $2$ to the vertex $v$.
It is easy to see that there exist multiset distinguishing $2$-edge colourings
of such graph, none of which is locally irregular.
Creating the later colouring requires $3$ colours.
This example may also be easily generalized by
substituting the paths of length $2$ with any other even paths.

In the following, given two graphs $H_1=(V_1,E_1)$, $H_2=(V_2,E_2)$, usually subgraphs of a host graph $G$,
by $H_1\cup H_2$ we shall mean the graph $(V_1\cup V_2, E_1\cup E_2)$.
Moreover, we shall write $H_2\subset H_1$ if $V_2\subset V_1$ and $E_2\subset E_1$,
and in case of $H_2\subset H_1$, we shall also write $H_1-E(H_2)$ to denote the graph obtained from $H_1$ by removing the edges of $H_2$.
Given a subset $E'$ of edges of a graph $G=(V,E)$,
the \emph{graph induced by} $E'$ shall be understood as $G':=(V,E')$.

\section{Tools}\label{section_with_tools}

We shall use
the Lov\'asz Local Lemma and the Chernoff Bound,
classical
tools of the probabilistic method,
see e.g.~\cite{AlonSpencer} and~\cite{MolloyReed}, respectively.
\begin{theorem}[\textbf{The Local Lemma; General Case}]\label{LLL-general}
Let $\mathcal{A}$ be a finite family of (typically bad) events in any probability space
and let $D=(\mathcal{A},E)$ be a directed graph
such that every event $A\in \mathcal{A}$ is mutually independent of all the events $\{B: (A,B)\notin E\}$.
Suppose that there are real numbers $x_A$ ($A\in\mathcal{A}$) such that for every $A\in\mathcal{A}$, $0\leq x_A<1$ and
\begin{equation}\label{EqLLL-general}
{\rm \emph{\textbf{Pr}}}(A) \leq x_A \prod_{B\leftarrow A} (1-x_B).
\end{equation}
Then ${\rm \emph{\textbf{Pr}}}(\bigcap_{A\in\mathcal{A}}\overline{A})>0$.
\end{theorem}
Here $B\leftarrow A$ (or $A\rightarrow B$) means that there is an arc from $A$ to $B$ in $D$,
the so called \emph{dependency digraph}.

\begin{theorem}[\textbf{Chernoff Bound}]\label{ChernofBoundTh}
For any $0\leq t\leq np$:
$${\rm \emph{\textbf{Pr}}}(|{\rm BIN}(n,p)-np|>t)<2e^{-\frac{t^2}{3np}},$$
where ${\rm BIN}(n,p)$ is the sum of $n$ independent variables, each equal to $1$ with probability $p$ and $0$ otherwise.
\end{theorem}

For the deterministic part of our proof we shall in turn use a consequence of the following theorem
from~\cite{Louigi30}
(see also \cite{Louigi2,Louigi} for similar degree theorems and their applications).
\begin{theorem} \label{1_3_2_3Lemma}
Suppose that for some graph $G=(V,E)$ we have chosen, for every
vertex $v$, two integers:
$$a^-_v\in \left[\frac{d(v)}{3}-1,\frac{d(v)}{2}\right],
~~a^+_v\in \left[\frac{d(v)}{2}-1,\frac{2d(v)}{3}\right].$$
Then there exists a spanning subgraph $H$ of $G$ such that for every $v\in V$:
$$d_H(v)\in\{a^-_v,a^-_v+1,a^+_v,a^+_v+1\}.$$
\end{theorem}

\begin{corollary} \label{1_6Lemma}
Suppose that for some graph $G=(V,E)$ we have chosen, for every
vertex $v$, a positive integer $\lambda_v$ with $6\lambda_v \leq d(v)$. Then for every
assignment
$$t:V\to \mathbb{Z},$$
there exists a spanning subgraph $H$ of $G$ such that $d_H(v)\in[\frac{d(v)}{3},\frac{2d(v)}{3}]$ and
$d_H(v)\equiv t(v) \pmod {\lambda_v}$ or $d_H(v)\equiv t(v)+1 \pmod {\lambda_v}$ for each $v\in V$.
\end{corollary}

\begin{proof}
For every vertex $v\in V$ we have:
$$\left\lfloor\frac{d(v)}{2}\right\rfloor - \left\lfloor\frac{d(v)}{3}\right\rfloor+1
\geq \frac{d(v)-1}{2} - \frac{d(v)}{3}+1 > \frac{d(v)}{6}\geq \lambda_v.$$
Since both sides of the inequality above are integers, then in fact:
$$\left\lfloor\frac{d(v)}{2}\right\rfloor - \left\lfloor\frac{d(v)}{3}\right\rfloor \geq \lambda_v.$$
Analogously,
$$\left\lfloor\frac{2d(v)}{3}\right\rfloor - \left\lfloor\frac{d(v)}{2}\right\rfloor+1
\geq \frac{2d(v)-2}{3}-\frac{d(v)}{2}+1 > \frac{d(v)}{6}\geq \lambda_v,$$
hence,
$$\left\lfloor\frac{2d(v)}{3}\right\rfloor - \left\lfloor\frac{d(v)}{2}\right\rfloor \geq \lambda_v.$$
Thus the sets of integers
$$\left\{\left\lfloor\frac{d(v)}{3}\right\rfloor+1, \ldots, \left\lfloor\frac{d(v)}{2}\right\rfloor  \right\}
~~{\rm and}~~
\left\{\left\lfloor\frac{d(v)}{2}\right\rfloor, \ldots, \left\lfloor\frac{2d(v)}{3}\right\rfloor-1  \right\}$$
both contain all remainders modulo $\lambda_v$.
The thesis follows then by Theorem~\ref{1_3_2_3Lemma}
(it is sufficient to choose $a^-_v,a^+_v$ from these sets, resp., so that $a^-_v,a^+_v\equiv t(v) \pmod {\lambda_v}$). $\Box$
\end{proof}

\section{Proof of Theorem~\ref{main_Th_general_graphs}}
\subsection{Notions}
Let $G=(V,E)$ be a graph of minimum degree $\delta\geq 10^{10}$.
In the following by $d(v)$ we shall mean the degree of a vertex $v$ in $G$,
and we shall write $d(v)^p$ for short instead of $(d(v))^p$.
Let us denote the auxiliary `optimizing' constant $$\beta:=2^{\frac{1}{0.38}},$$
where $\beta\approx 6.2$ ($6.19<\beta< 6.2$).
In order to apply Corollary~\ref{1_6Lemma}, we shall also need two auxiliary vertex labelings, say $c_1$ and $c_2$,
with certain regular features.
Thus for every vertex $v$ let us first randomly and independently choose one value in $\{0,1,\ldots, 2^{\lceil\log_\beta d(v)\rceil}-1\}$,
each with equal probability, and denote it by $c_1(v)$.
Then let us independently repeat our drawing, i.e., again for every $v\in V$
randomly and independently we choose one value in $\{0,1,\ldots,2^{\lceil\log_\beta d(v)\rceil}-1\}$,
each with equal probability, and denote it by $c_2(v)$.

By our construction below it shall be clear that every edge whose one end has
the degree at least $\beta$ times bigger than the other will be `safe' from any potential conflicts between its end-vertices.
Some of the remaining edges shall require extra attention though, and shall thus be called `risky'.
We distinguish three kinds of these, i.e., we say that an edge $uv$ with $1/\beta d(v)< d(u)< \beta d(v)$ is:
\begin{itemize}
  \item \emph{risky of type 1} if\\
$2^{\lceil\log_\beta d(u)\rceil}c_1(u)\equiv 2^{\lceil\log_\beta d(v)\rceil}c_1(v)
\pmod{\min \{4^{\lceil\log_\beta d(u)\rceil},4^{\lceil\log_\beta d(v)\rceil} \}}$;
\item \emph{risky of type 2} if\\
$2^{\lceil\log_\beta d(u)\rceil}c_2(u)\equiv 2^{\lceil\log_\beta d(v)\rceil}c_2(v)
\pmod{\min \{4^{\lceil\log_\beta d(u)\rceil},4^{\lceil\log_\beta d(v)\rceil} \}}$;
\item \emph{risky of type 3} if\\
$|d(u)-3\cdot 2^{\lceil\log_\beta d(u)\rceil}(c_1(u)+c_2(u))-d(v)+3\cdot 2^{\lceil\log_\beta d(v)\rceil}(c_1(v)+c_2(v))|\\
<\min\{3\cdot 2^{\lceil\log_\beta d(u)\rceil},3\cdot 2^{\lceil\log_\beta d(v)\rceil}\} \pmod{\min \{3\cdot 4^{\lceil\log_\beta d(u)\rceil},3\cdot 4^{\lceil\log_\beta d(v)\rceil} \}}$,
\end{itemize}
where, given integers $b$ and $k$ with $b\in\{1,\ldots,k\}$, by writing
$|a|<b\pmod k$ we mean that $a$ is an integer which is congruent to
one of the following: $-b+1,-b+2,\ldots,b-1$ modulo $k$.
Denote the sets of risky edges of types 1, 2 and 3 by $R_1$, $R_2$ and $R_3$, respectively.
For each $v\in V$, let us also denote:
\begin{eqnarray*}
A(v):&=&\left\{u\in N_G(v): uv\in R_1 \right\},\\
B(v):&=&\left\{u\in N_G(v): uv\in R_2 \right\},\\
C(v):&=&\left\{u\in N_G(v): uv\in R_3 \right\},\\
F(v):&=&B(v)\cap C(v)=\left\{u\in N_G(v): uv\in R_2\cap R_3 \right\}.
\end{eqnarray*}

\subsection{Probabilistic Lemma}
\begin{claim}
With positive probability, for every vertex $v\in V$:
\begin{eqnarray}
|A(v)|,|B(v)|,|C(v)| &\leq& 8d(v)^{0.62}~~{\rm and}\label{ABC_ineq}\\
|F(v)| &\leq& 12d(v)^{0.24}.\label{F_ineq}
\end{eqnarray}
\end{claim}
\begin{proof}
For every $v\in V$, let $X_v,Y_v,Z_v,T_v$ be the random variables of the cardinalities of the sets $A(v),B(v),C(v),F(v)$, resp.,
and let $A_v$, $B_v$, $C_v$, $F_v$ denote the events that $X_v>8d(v)^{0.62}$, $Y_v>8d(v)^{0.62}$,
$Z_v>8d(v)^{0.62}$ and $T_v>12d(v)^{0.24}$, respectively.
Consider a vertex $v\in V$, and let $u$ be any of its neighbours
with $d(u)\in(1/\beta d(v),\beta d(v))$.
Note that $\lceil\log_\beta d(u)\rceil\in\left\{\lceil\log_\beta d(v)\rceil-1, \lceil\log_\beta d(v)\rceil, \lceil\log_\beta d(v)\rceil+1\right\}$.
Then for arbitrarily fixed $c^*_1,c^*_2,c^*_3\in \{0,1,\ldots, 2^{\lceil\log_\beta d(v)\rceil}-1\}$, in all cases:
\begin{eqnarray}
{\rm \textbf{Pr}}\left(u\in A(v)|c_1(v)=c^*_1\right) &\leq& \frac{1}{2^{\lceil\log_\beta d(v)\rceil-1}} \leq \frac{1}{2^{\log_\beta d(v)-1}}\nonumber\\ &=&\frac{2}{d(v)^{\frac{1}{\log_2 \beta}}} =  \frac{2}{d(v)^{0.38}},\label{A(v)_bound}
\end{eqnarray}
\begin{eqnarray}
{\rm \textbf{Pr}}\left(u\in B(v)|c_2(v)=c^*_2\right) &\leq& \frac{1}{2^{\lceil\log_\beta d(v)\rceil-1}} \leq \frac{1}{2^{\log_\beta d(v)-1}}\nonumber\\ &=&\frac{2}{d(v)^{\frac{1}{\log_2 \beta}}} =  \frac{2}{d(v)^{0.38}},\label{B(v)_bound}
\end{eqnarray}
\begin{eqnarray}
{\rm \textbf{Pr}}\left(u\in C(v)|c_1(v)=c^*_1\wedge c_2(v)=c^*_2\wedge c_2(u)=c^*_3\right) &\leq& \frac{2}{2^{\lceil\log_\beta d(v)\rceil-1}}\nonumber\\
&\leq& \frac{4}{d(v)^{\frac{1}{\log_2 \beta}}}\nonumber\\
&=& \frac{4}{d(v)^{0.38}},\label{C(v)_bound}
\end{eqnarray}
hence by the total probability:
\begin{equation}
{\rm \textbf{Pr}}\left(u\in C(v)|c_1(v)=c^*_1\wedge c_2(v)=c^*_2\right) \leq \frac{4}{d(v)^{0.38}}. \label{C(v)_bound_bis}
\end{equation}
Finally, since the choices for $c_1$ and $c_2$ are independent,
by~(\ref{B(v)_bound}) and~(\ref{C(v)_bound}),
\begin{eqnarray}
&&{\rm \textbf{Pr}}\left(u\in F(v)|c_1(v)=c^*_1\wedge c_2(v)=c^*_2\right)\nonumber\\
&=&{\rm \textbf{Pr}}\left(u\in C(v)| c_1(v)=c^*_1\wedge c_2(v)=c^*_2 \wedge u\in B(v)\right)\nonumber\\
&\times & {\rm \textbf{Pr}}\left(u\in B(v)|c_1(v)=c^*_1\wedge c_2(v)=c^*_2\right)\nonumber\\
&\leq& \frac{4}{d(v)^{0.38}} \cdot \frac{2}{d(v)^{0.38}} = \frac{8}{d(v)^{0.76}}.\label{F(v)_bound}
\end{eqnarray}

Consequently,
since all choices are independent and $2/d(v)^{0.38}\leq 4/d(v)^{0.38}$, by~(\ref{A(v)_bound})
and the Chernoff Bound we obtain:
\begin{eqnarray}
{\rm \textbf{Pr}}\left(A_v|c_1(v)=c^*_1\right) &=&  {\rm \textbf{Pr}}\left(X_v > 8d(v)^{0.62}|c_1(v)=c^*_1\right)\nonumber\\
&\leq& {\rm \textbf{Pr}}\left({\rm BIN}\left(d(v),\frac{4}{d(v)^{0.38}}\right) > 8d(v)^{0.62}\right)\nonumber\\
&\leq& {\rm \textbf{Pr}}\left(\left|{\rm BIN}\left(d(v),\frac{4}{d(v)^{0.38}}\right) - 4d(v)^{0.62}\right| > 4d(v)^{0.62}\right)\nonumber\\
&<& 2 e^{-\frac{4d(v)^{0.62}}{3}}. \nonumber
\end{eqnarray}
By the total probability we thus obtain that:
\begin{equation}\label{A_v_bound}
{\rm \textbf{Pr}}\left(A_v\right) < 2 e^{-\frac{4d(v)^{0.62}}{3}}.
\end{equation}

Analogously, by~(\ref{B(v)_bound}) and (\ref{C(v)_bound_bis}),
\begin{equation}\label{B_v_C_v_bounds}
{\rm \textbf{Pr}}\left(B_v\right) < 2 e^{-\frac{4d(v)^{0.62}}{3}} ~~{\rm and}~~ {\rm \textbf{Pr}}\left(C_v\right) < 2 e^{-\frac{4d(v)^{0.62}}{3}}.
\end{equation}

Finally, again by the Chernoff Bound and~(\ref{F(v)_bound}):
\begin{eqnarray}
{\rm \textbf{Pr}}\left(F_v|c_1(v)=c^*_1 \wedge c_2(v)=c^*_2\right) &=&  {\rm \textbf{Pr}}\left(T_v > 12d(v)^{0.24} |c_1(v)=c^*_1 \wedge c_2(v)=c^*_2\right)\nonumber\\
&\leq& {\rm \textbf{Pr}}\left({\rm BIN}\left(d(v),\frac{8}{d(v)^{0.76}}\right) > 12d(v)^{0.24}\right)\nonumber\\
&\leq& {\rm \textbf{Pr}}\left(\left|{\rm BIN}\left(d(v),\frac{8}{d(v)^{0.76}}\right) - 8d(v)^{0.24}\right| > 4d(v)^{0.24}\right)\nonumber\\
&<& 2 e^{-\frac{2d(v)^{0.24}}{3}},\nonumber
\end{eqnarray}
and hence, by the total probability,
\begin{equation}\label{F_v_bound}
{\rm \textbf{Pr}}\left(F_v\right) < 2 e^{-\frac{2d(v)^{0.24}}{3}}.
\end{equation}

Note that for every vertex $v\in V$, the events
$A_v$, $B_v$, $C_v$ and $F_v$ depend only on the
random choices for $v$ and its adjacent vertices $u$ with
$1/\beta d(v)< d(u)< \beta d(v)$.
Thus each of these events (corresponding to the vertex $v$)
is mutually independent of all
events except (possibly) for
these corresponding to the vertex $v$ itself, those corresponding to the neighbours $v'$ of $v$ with $1/\beta d(v)< d(v')< \beta d(v)$
and those corresponding to the neighbours $v''$ of such $v'$ for which
$1/\beta d(v')< d(v'')< \beta d(v')$.
In order to construct a dependency digraph $D$ necessary to apply Theorem~\ref{LLL-general},
from each of the events $A_v$, $B_v$, $C_v$, $F_v$,
we draw arrows pointing at all other events
corresponding to the vertices $w$ ($w=v$ or $w=v'$ or $w=v''$) with the properties described above, for $v\in V$.
For any event $L$ corresponding to a vertex $v$ of degree $d$ in $G$ (i.e., $L=A_v$, $L=B_v$, $L=C_v$ or $L=F_v$),
we then set
\begin{equation}\label{auxiliary_for_LLL_0}
x_L=\frac{1}{1+d^3}.
\end{equation}
By our construction, for every such $L$,
\begin{equation}\label{auxiliary_for_LLL_1}
d^+_D(L)\leq 3+4d+4d(\lfloor \beta d\rfloor-1) = 3+4d\lfloor \beta d\rfloor,
\end{equation}
where $d^+_D(L)$ is the out-degree of $L$ in $D$.
Moreover, if $L\rightarrow Q$ in $D$, then $Q$ is an event corresponding to a vertex $w$
with
\begin{equation}\label{auxiliary_for_LLL_2}
\frac{1}{\beta^2} d< d(w)< \beta^2 d.
\end{equation}
By (\ref{auxiliary_for_LLL_0}), (\ref{auxiliary_for_LLL_1}) and (\ref{auxiliary_for_LLL_2}),
since $\frac{x}{1+x}>e^{-\frac{1}{x}}$ for $x>0$, we thus obtain:
\begin{eqnarray}
x_{L} \prod_{Q\leftarrow L} (1-x_Q) &=& \left[x_L\cdot\frac{1}{1-x_L}\right]\cdot\left[(1-x_L)\prod_{Q\leftarrow L} (1-x_Q)\right]\nonumber\\
&>& \left[\frac{1}{1+d^3}\cdot\frac{1}{1-\frac{1}{1+d^3}}\right] \left[\left(1-\frac{1}{1+(\frac{d}{\beta^2})^3}\right)^{1+(3+4d\lfloor \beta d\rfloor)}\right]\nonumber\\
&\geq&\frac{1}{d^3} \left(\frac{(\frac{d}{\beta^2})^3}{1+(\frac{d}{\beta^2})^3}\right)^{25d^2}
> \frac{1}{d^3} e^{-\frac{1}{(\frac{d}{\beta^2})^3}25d^2}
\nonumber\\
&=& \left(2\frac{1}{2d^3}\right) \left(e^{-\frac{25\cdot \beta^6}{d}}\right) > \left(2 e^{-\frac{d^{0.24}}{3}}\right)
\left(e^{-\frac{d^{0.24}}{3}}\right) = 2 e^{-\frac{2d^{0.24}}{3}},\label{LLL_appl_main_ineq}
\end{eqnarray}
where $e^{-\frac{25\cdot \beta^6}{d}}\geq e^{-\frac{d^{0.24}}{3}}$ for $d\geq \left(3\cdot 25 \cdot \beta^6\right)^\frac{1}{1.24}\approx 221,460$,
while $\frac{1}{2d^3}>e^{-\frac{d^{0.24}}{3}}$ is equivalent to the inequality
$$f(d):=\frac{d^{0.24}}{3}-\ln(2d^3)>0,$$
which holds e.g. for $d\geq 10^{10}$, since $f(10^{10})\approx 14>0$ and
$$f'(d)=\frac{0.08}{d^{0.76}}-\frac{3}{d}>0$$
for $d>\left(\frac{3}{0.08}\right)^\frac{1}{0.24}\approx 3,617,959$.

By (\ref{A_v_bound}), (\ref{B_v_C_v_bounds}), (\ref{F_v_bound}), (\ref{LLL_appl_main_ineq}) and the Local Lemma we thus obtain that
$${\rm \textbf{Pr}}\left(\bigcap_{v\in V}\overline{A_v}\cap \overline{B_v}\cap \overline{C_v}\cap \overline{F_v}\right)>0.~~\Box$$
\end{proof}

\subsection{Construction}
Suppose then that we have chosen the labelings $c_1$ and $c_2$ so that (\ref{ABC_ineq}) and (\ref{F_ineq}) hold for every $v\in V$.
We shall use twice the fact that
\begin{equation}\label{double_use_ineq}
\frac{d}{3}-16d^{0.62}>72d^{0.76}
\end{equation}
for $d\geq 10^{10}$. Indeed, if only $d> 16^{\frac{1}{0.14}}\approx 398,893,555$, then the left hand side of inequality~(\ref{double_use_ineq}) is greater than
$\frac{d}{3}-d^{0.76}$, what in turn is greater than the right hand side of inequality~(\ref{double_use_ineq})
if only $d>(3\cdot 73)^\frac{1}{0.24}\approx 5,647,425,084$.

Let us temporarily remove from $G$ all risky edges
of type 1
and denote the graph obtained by $G'$. By~(\ref{ABC_ineq}) and~(\ref{double_use_ineq}), for every vertex $v\in V$,
\begin{eqnarray}
d_{G'}(v)&\geq& d(v)-8d(v)^{0.62} \geq 72\cdot d(v)^{0.76}\label{G'/6}\\
&=& 72\cdot 4^{\log_\beta d(v)} \geq  6\cdot 3\cdot 4^{\lceil\log_\beta d(v)\rceil}.\nonumber
\end{eqnarray}
By (\ref{G'/6}) and  Corollary~\ref{1_6Lemma}, we may thus find a subgraph $H_1$ of $G'$ such that
$d_{H_1}(v)$ has one of the two
remainders modulo $3\cdot 4^{\lceil\log_\beta d(v)\rceil}$, namely:
\begin{equation}\label{dH1v}
d_{H_1}(v)\equiv 3\cdot 2^{\lceil\log_\beta d(v)\rceil}c_1(v),3\cdot 2^{\lceil\log_\beta d(v)\rceil}c_1(v)+1 \pmod {3\cdot 4^{\lceil\log_\beta d(v)\rceil}}
\end{equation}
for every $v\in V$, and (by~(\ref{G'/6})):
\begin{equation}\label{degrees_in_H_1}
d_{H_1}(v) \in \left[\frac{d_{G'}(v)}{3},\frac{2d_{G'}(v)}{3}\right]\subset
\left[\frac{d(v)-8d(v)^{0.62}}{3},\frac{2d(v)}{3}\right].
\end{equation}
We paint the edges of $H_1$ with colour $1$, and claim that $H_1$ is locally irregular.
Consider an edge $uv\in E(H_1)$. By our construction, $uv\notin R_1$.
Condition~(\ref{dH1v}) implies that
either $d_{H_1}(u)$ or $d_{H_1}(u)-1$ must be a multiplicity of
$\min\{3\cdot 2^{\lceil\log_\beta d(u)\rceil},3\cdot 2^{\lceil\log_\beta d(v)\rceil}\}$,
and similarly, either $d_{H_1}(v)$ or $d_{H_1}(v)-1$ is a multiplicity of the same quantity.
If $1/\beta <d(u)<\beta d(v)$, then
these multiplicities cannot however be the same, since otherwise $uv$ would have to be a risky edge of type 1. Therefore,
$d_{H_1}(u)\neq d_{H_1}(v)$ in such a case.
The same holds also for the (remaining) edges with a greater spread between the degrees of the end-vertices.
We shall exhibit that in the next subsection.

Denote by $G_1$ the graph obtained from $G$ by removing all (already painted) edges of $H_1$.
Let us (again temporarily) remove from $G_1$ all risky edges $e\in R_2\cup R_3$ of types 2 and 3,
and denote the graph obtained by $G''$. By (\ref{ABC_ineq}), (\ref{double_use_ineq}) and (\ref{degrees_in_H_1}),
for every vertex $v\in V$,
\begin{eqnarray}
d_{G''}(v)&\geq& \frac{d(v)}{3} -
16d(v)^{0.62} \geq 72\cdot d(v)^{0.76}\label{G''/6}\\
&=& 72\cdot 4^{\log_\beta d(v)}  \geq 6\cdot 3\cdot 4^{\lceil\log_\beta d(v)\rceil}.\nonumber
\end{eqnarray}

Let $C$ be the subgraph induced by these edges
of $G_1$ which belong to $R_3$,
hence $C$ and $G''$ are edge-disjoint.
For every $v\in V$, denote by
\begin{equation}\label{c_v}
c_v:=d_{C}(v)=|C(v)\cap N_{G_1}(v)|
\end{equation}
the number of risky edges
of type 3 incident with $v$ in $G_1$.

Subsequently consider the subgraph $F$ induced by these edges
of $G_1$ which belong to $R_2\cap R_3$.
Note that $F\subset C$.
By~(\ref{F_ineq}), for every vertex $v\in V$,
$$d_{F}(v)\leq  12d(v)^{0.24} < \frac{1}{2}d(v)^{0.38},$$
where the second inequality holds since $d(v)>24^\frac{1}{0.14}\approx 7,221,904,256$, and thus
$$d_{F}(v)<  \frac{1}{2}d(v)^{\frac{1}{\log_2 \beta}} = 2^{\log_\beta d(v)-1} \leq 2^{\lceil\log_\beta d(v)\rceil-1}.$$
Since the left-hand-side and the right-hand-side of the inequality above are both integers,
we in fact obtain that $d_{F}(v)\leq 2^{\lceil\log_\beta d(v)\rceil-1}-1$.
Hence, we may greedily find a \emph{proper} vertex colouring
$$h:V\to\{0,1,\ldots, 2^{\lceil\log_\beta \Delta(G)\rceil-1}-1\}$$
of $F$ so that
\begin{equation}\label{h(v)_bound}
h(v)\leq 2^{\lceil\log_\beta d(v)\rceil-1}-1
\end{equation}
for every $v\in V$.

By (\ref{G''/6}) and Corollary~\ref{1_6Lemma},
we then may find a subgraph $H_2$ of $G''$ such that
\begin{eqnarray}
d_{H_2}(v)&\equiv& 3\cdot 2^{\lceil\log_\beta d(v)\rceil} c_2(v)+3h(v)-c_v,\nonumber\\
&&3\cdot 2^{\lceil\log_\beta d(v)\rceil} c_2(v)+3h(v)-c_v+1\pmod {3\cdot 4^{\lceil\log_\beta d(v)\rceil}} \label{dH2v}
\end{eqnarray}
for every $v\in V$, and (by~(\ref{ABC_ineq}))
\begin{eqnarray}
d_{H_2}(v) &\in& \left[\frac{d_{G''}(v)}{3},\frac{2d_{G''}(v)}{3}\right]\nonumber\\
&\subset& \left[\frac{d(v)-d_{H_1}(v)-16d(v)^{0.62}}{3},\frac{2(d(v)-d_{H_1}(v))}{3}\right].\label{degrees_in_H_2}
\end{eqnarray}
Then we denote $H'_2:=H_2\cup C$, $H'_3:=G-E(H_1\cup H'_2)$, and colour
the edges of the first of these graphs with $2$, while the edges of the second one with $3$.
Note that since $H'_3=G_1-E(H_2\cup C)$, then $H'_3$ contains no risky edges of type 3.

\subsection{Validity}
Since $H_2$ and $C$ are edge-disjoint, by~(\ref{c_v}) and~(\ref{dH2v}),
\begin{eqnarray}
d_{H'_2}(v)&\equiv& 3\cdot 2^{\lceil\log_\beta d(v)\rceil} c_2(v)+3h(v),\nonumber\\
&&3\cdot 2^{\lceil\log_\beta d(v)\rceil} c_2(v)+3h(v)+1\pmod {3\cdot 4^{\lceil\log_\beta d(v)\rceil}}\label{dH'2v}
\end{eqnarray}
for every $v\in V$.
Consider an edge $uv\in E(H'_2)$ with $1/\beta <d(u)<\beta d(v)$.
Then either $d_{H'_2}(u)-3h(u)$ or $d_{H'_2}(u)-3h(u)-1$ must be a multiplicity of
$\min\{3\cdot 2^{\lceil\log_\beta d(u)\rceil},3\cdot 2^{\lceil\log_\beta d(v)\rceil}\}\geq \max
\{3\cdot 2^{\lceil\log_\beta d(u)\rceil-1},3\cdot 2^{\lceil\log_\beta d(v)\rceil-1}\}$,
and similarly, either $d_{H'_2}(v)-3h(v)$ or $d_{H'_2}(v)-3h(v)-1$ is a multiplicity of the same quantity.
If $uv\notin R_2$, then analogously as above these multiplicities cannot be the same, hence because
by~(\ref{h(v)_bound}),
$0\leq 3h(u),3h(v)\leq \max
\{3\cdot 2^{\lceil\log_\beta d(u)\rceil-1},3\cdot 2^{\lceil\log_\beta d(v)\rceil-1}\}-3$,
we obtain that
$d_{H'_2}(u)\neq d_{H'_2}(v)$.
Otherwise, by our construction, $uv\in R_2\cap R_3$ is an edge of $F$, and hence
$d_{H'_2}(u)\neq d_{H'_2}(v)$ by the `properness' of $h$, since $0\leq 3h(u),3h(v)\leq \max
\{3\cdot 2^{\lceil\log_\beta d(u)\rceil-1},3\cdot 2^{\lceil\log_\beta d(v)\rceil-1}\}-3$.

By~(\ref{dH1v}) and~(\ref{dH'2v}),
\begin{eqnarray}
d_{H'_3}(v)&\equiv& d(v)-3\cdot 2^{\lceil\log_\beta d(v)\rceil} (c_1(v)+c_2(v))-3h(v),\nonumber\\
&&d(v)-3\cdot 2^{\lceil\log_\beta d(v)\rceil} (c_1(v)+c_2(v))-3h(v)-1,\nonumber\\
&&d(v)-3\cdot 2^{\lceil\log_\beta d(v)\rceil} (c_1(v)+c_2(v))-3h(v)-2\pmod {3\cdot 4^{\lceil\log_\beta d(v)\rceil}}\label{dH1v+dH'2v}
\end{eqnarray}
for every $v\in V$.
Then, similarly as above, $d_{H'_3}(u)\neq d_{H'_3}(v)$ for every edge $uv$ of $H'_3$ with $1/\beta d(v)<d(u)<\beta d(v)$,
since then $uv\notin R_3$ by our construction.

As for the remaining edges of $H_1$, $H'_2$ and $H'_3$, let us first note that
for every vertex $v\in V$, by~(\ref{degrees_in_H_1}),
\begin{equation}\label{degrees_in_H_1_bis}
d_{H_1}(v)\in
\left[\frac{1}{3}d(v)-\frac{8}{3}d(v)^{0.62},\frac{2}{3}d(v)\right],
\end{equation}
hence, by~(\ref{ABC_ineq}), (\ref{degrees_in_H_2}), (\ref{degrees_in_H_1_bis}) and our construction,
\begin{eqnarray}
d_{H'_2}(v)&\in&
\left[\frac{1}{3}d(v)-\frac{1}{3}d_{H_1}(v)-\frac{16}{3}d(v)^{0.62},
\frac{2}{3}d(v)-\frac{2}{3}d_{H_1}(v)+8d(v)^{0.62}\right]\label{degrees_in_H_2_bis2}\\
&\subset& \left[\frac{1}{3}d(v)-\frac{1}{3}\cdot\frac{2}{3}d(v)-\frac{16}{3}d(v)^{0.62},\right.\nonumber\\ & & \left.
\frac{2}{3}d(v)-\frac{2}{3}\left(\frac{1}{3}d(v)-\frac{8}{3}d(v)^{0.62}\right)+8d(v)^{0.62}\right]\nonumber\\
&=&\left[\frac{1}{9}d(v)-\frac{16}{3}d(v)^{0.62}, \frac{4}{9}d(v)+\frac{88}{9}d(v)^{0.62}\right],\label{degrees_in_H_2_bis}
\end{eqnarray}
and thus, by (\ref{degrees_in_H_1_bis}) and (\ref{degrees_in_H_2_bis2}),
\begin{eqnarray}
d_{H'_3}(v)&\in&\left[d(v)-d_{H_1}(v)-\left(\frac{2}{3}d(v)-\frac{2}{3}d_{H_1}(v)+8d(v)^{0.62}\right)\right.,\nonumber\\
& & \left.d(v)-d_{H_1}(v)-\left(\frac{1}{3}d(v)-\frac{1}{3}d_{H_1}(v)-\frac{16}{3}d(v)^{0.62}\right)\right]\nonumber\\
&=& \left[\frac{1}{3}d(v)-\frac{1}{3}d_{H_1}(v)-8d(v)^{0.62},
\frac{2}{3}d(v)-\frac{2}{3}d_{H_1}(v)+\frac{16}{3}d(v)^{0.62}\right]\nonumber\\
&\subset& \left[\frac{1}{3}d(v)-\frac{1}{3}\cdot \frac{2}{3}d(v)-8d(v)^{0.62},\right.\nonumber\\
&&\left.\frac{2}{3}d(v)-\frac{2}{3}\left(\frac{1}{3}d(v)-\frac{8}{3}d(v)^{0.62}\right)+\frac{16}{3}d(v)^{0.62}\right]\nonumber\\
&=&\left[\frac{1}{9}d(v)-8d(v)^{0.62},
\frac{4}{9}d(v)+\frac{64}{9}d(v)^{0.62}\right].\label{degrees_in_H_3_bis}
\end{eqnarray}
Consequently, since $\frac{4}{9}d+\frac{88}{9}d^{0.62}\leq \frac{2}{3}d$ for $d\geq 44^{\frac{1}{0.38}}\approx 21,129$,
by (\ref{degrees_in_H_1_bis}), (\ref{degrees_in_H_2_bis}) and (\ref{degrees_in_H_3_bis}),
$$d_{H_1}(v),d_{H'_2}(v),d_{H'_3}(v)\in\left[\frac{1}{9}d(v)-8d(v)^{0.62},\frac{2}{3}d(v)\right],$$
and thus
$$d_{H_1}(v),d_{H'_2}(v),d_{H'_3}(v)\in\left[\frac{4}{37}d(v),\frac{2}{3}d(v)\right],$$
since $\frac{1}{9}d-8d^{0.62}\geq \frac{4}{37}d$ for $d\geq (8\cdot 333)^{\frac{1}{0.38}}\approx 1,034,102,857$.
Since $\frac{2}{3}/\frac{4}{37}< 6.17 <\beta$, this guarantees that if (without the loss of generality) $d(u)\geq \beta d(v)$
for some edge $uv\in E$, then $d_{H_1}(u)\neq d_{H_1}(v)$, $d_{H'_2}(u)\neq d_{H'_2}(v)$ and $d_{H'_3}(u)\neq d_{H'_3}(v)$.
All subgraphs $H_1$, $H'_2$ and $H'_3$ are thus locally irregular.
$\Box$

\section{Concluding remarks}
Note that Conjecture~\ref{BBPWConjecture1} still remains open.
It would be interesting to settle it at least for bipartite graphs.
\begin{problem}\label{JP_Question_1}
Can every connected bipartite graph which is not an odd length path be decomposed into three locally irregular subgraphs?
\end{problem}
Moreover, thus far it is not even known if any finite number of locally irregular subgraphs admitted is sufficient in general.
\begin{problem}\label{JP_Question_2}
Does there exist a constant $K$ such that
every connected graph which
does not belong to $\mathfrak{T}$ and is not an odd length path nor an odd length cycle can be decomposed into (at most) $K$ locally irregular subgraphs?
\end{problem}
This is not known in the case of bipartite graphs either.



\begin{thebibliography}{99}
\bibitem{Louigi30}
L. Addario-Berry, K. Dalal, C. McDiarmid, B.A. Reed, A. Thomason,
\emph{Vertex-Colouring Edge-Weightings}, Combinatorica, 27(1) (2007), pp. 1--12.

\bibitem{Louigi2}
L. Addario-Berry, R.E.L. Aldred, K.
Dalal, B.A. Reed, \emph{Vertex colouring edge partitions}, J.
Combin. Theory Ser. B, 94(2) (2005), pp. 237--244.

\bibitem{Louigi}
L. Addario-Berry, K. Dalal, B.A. Reed, \emph{Degree Constrained
Subgraphs},
Discrete Appl. Math., 156(7) (2008), pp. 1168--1174.

%
%
%
\bibitem{AlonSpencer}
N. Alon, J.H. Spencer, \emph{The Probabilistic Method, 2nd edition}, Wiley, New York,
2000.

\bibitem{LocalIrreg_1}
O. Baudon, J. Bensmail, J. Przyby{\l}o, M. Wo\'zniak, \emph{On decomposing regular graphs into locally irregular subgraphs}, European J. Combin. 49 (2015) 90--104.

%
%
\bibitem{Diestel}
R. Diestel, \emph{Graph Theory, Electronic Edition 2005},
Springer, New York, 2005.

%
%
%
%
%
\bibitem{KalKarPf_123}
M. Kalkowski, M. Karo\'nski, F. Pfender, \emph{Vertex-coloring edge-weightings: Towards the 1-2-3 conjecture}, J. Combin. Theory Ser. B, 100 (2010), pp. 347--349.

\bibitem{123KLT}
M. Karo\'nski, T. \L uczak, A. Thomason, \emph{Edge weights and
vertex colours}, J. Combin. Theory Ser. B, 91 (2004), pp. 151--157.

%
%
\bibitem{MolloyReed}
M. Molloy, B. Reed, \emph{Graph Colouring and the Probabilistic Method}, Springer, Berlin, 2002.

%
%
%
%
\end{thebibliography}
\end{document}